\documentclass[reqno,12pt]{amsart}

\usepackage{amsmath,amssymb,latexsym}

\newtheorem{te}{Theorem}[section]
\newtheorem{lema}{Lemma}[section]

\newtheorem{cor}{Corollary}[section]

\theoremstyle{definition}

\newtheorem{de}{Definition}[section]
\theoremstyle{remark}
\newtheorem{rem}{Remark}[section]

\newcommand{\sing}{\operatorname{singsupp}}

\def\cN{\mathcal{N}}

\def\R{\mathbb{R}}

\def\N{\mathbb{N}}

\def\la{\langle}
\def\ra{\rangle}

\def\Ren{\mathbb{R}^d}
\def\Renn{\mathbb{R}^{2d}}

\def\cS{\mathcal{S}}

\def\o{\omega}
\def\cM{\mathcal{M}}

\def\inv{^{-1}}

\def\Mmpq{M_m^{p,q}}
\def\phas{(x,\omega )}

\def\Fur{\mathcal{F}}

\def\be{\begin{equation}}
\def\ee{\end{equation}}
\def\bena{\begin{eqnarray*}}
\def\ena{\end{eqnarray*}}

\def\mR{\mathbb{R}}   \def\mC{\mathbb{C}}
 \def\suml{\sum\limits} 
\def\dss{\displaystyle}

 \def\cS{\mathcal{S}}
 
 \def\D{\mathcal{D}}
 \def\E{\mathcal{E}}
 \def\Z{\mathbf{Z}_+}

\def\N{\mathbf{N}}

\def\lf {\lfloor}
\def\rf{\rfloor}

\def\phas{(x,\omega )}
\newcommand{\modsp}{modulation space}
\def\t{\tau}
\def\s{\sigma}

\def\cS{\mathcal{S}}

\DeclareMathOperator{\WF}{WF}

\theoremstyle{plain}
 \newtheorem{theorem}{Theorem}[section]
 \newtheorem{proposition}{Proposition}[section]
 \newtheorem{lemma}{Lemma}[section]
 
\theoremstyle{definition}
 
 \newtheorem{definition}{Definition}[section]
\theoremstyle{remark}
 \newtheorem{remark}{Remark}[section]
 \numberwithin{equation}{section}

\begin{document}

\title[Extended Gevrey regularity ...]
{Extended Gevrey regularity via  the short-time Fourier transform}

\author[Nenad Teofanov, Filip Tomi\'c]{Nenad Teofanov, Filip Tomi\'c}
\address{Nenad Teofanov\\ University of Novi Sad, Faculty of Sciences\\ Department
 of Mathematics and Informatics\\ Trg Dositeja Obradovi\' ca 4\\ 21000 Novi Sad \\ Serbia }
\email{nenad.teofanov@dmi.uns.ac.rs}
\address{Filip Tomi\'c\\ University of Novi Sad, Faculty of Technical Sciences\\ Department
 of Fundamental Sciences\\ Trg Dositeja Obradovi\' ca 6\\ 21000 Novi Sad \\ Serbia }
\email{filip.tomic@uns.ac.rs}



\begin{abstract}
We study the regularity of smooth functions whose derivatives are dominated by
sequences of the form $M_p^{\tau,\s}=p^{\tau p^{\s}}$, $\tau>0$, $\s\geq1$. We show that such functions can be characterized
through the decay properties of their short-time Fourier transforms (STFT),
and recover \cite[Theorem 3.1]{CNR} as the special case when  $ \t>1$ and $\s = 1$, i.e. when the Gevrey type regularity is considered.
These estimates lead to a Paley-Wiener type
theorem for extended Gevrey classes. In contrast to the related result
from \cite{PTT-05, PTT-04}, here we relax the assumption on compact support of the observed functions. Moreover, we introduce
the corresponding wave front set, recover it in terms of the STFT,
and discuss local regularity in such context.
\end{abstract}

\maketitle

\section{Introduction}

Classes of \emph{extended Gevrey functions} and the corresponding \emph{wave front sets} are introduced and investigated in \cite{PTT-01, PTT-02, PTT-03, TT}. Such classes consist of smooth function, and they are larger than any Gevrey class. This turned out to be important e.g. in the study of
strictly hyperbolic equations, see \cite{CL}.
Paley-Wiener type theorem for compactly supported extended Gevrey regular functions is given in \cite{PTT-05, PTT-04}, and it turns out that the Fourier-Laplace transform of such functions have certain logarithmic decay at infinity which can be expressed in terms of \emph{Lambert W function}. This fact is used to resolve the wave front sets in the context of extended Gevrey regularity. We refer to
\cite{PTT-02, PTT-03} for related theorems on propagation of singularities.

The aim of this paper is twofold. Firstly, we give another version of the Paley-Wiener theorem for extended Gevrey regularity and formulate the result  in terms of the \emph{short time Fourier transform} (STFT) (cf. \cite{Gro}). More precisely, we prove a generalization of \cite[Theorem 3.1]{CNR}, where the STFT estimates are related to Gevrey type regularity, and obtain the Paley-Wiener type result as its corollary. Secondly, we
give a description of (micro)local regularity related to the extended Gevrey regularity by the means of the STFT. This result is inspired by recent characterization of the $C^{\infty}$ wave front sets via the STFT, given in  \cite{PrangoskiPilip}.

The paper is organized as follows: In subsection \ref{notions} we fix some notation and in Section \ref{sec01} we collect the main notions and tools for our analysis:
Subsection \ref{secTest} contains basic facts concerning the extended Gevrey classes. In Subsection \ref{secAsocirana} we introduce the notion of \emph{extended associated function} which appears in the formulation of our main results. The correct asymptotic behavior of the
extended associated function is given by the means of the Lambert $W$ function, see Theorem \ref{propozicija}.
In subsection \ref{Modulation Spaces} we introduce the short-time Fourier transform
and modulation spaces defined by the means of decay and integrability conditions of the STFT of ultradistributions. We also recall some basic properties of modulation spaces.

In Section \ref{secGlavniRez} we prove Theorem \ref{glavnaTeoremaSTFT}. It is a generalization of \cite[Theorem 3.1]{CNR}  which turned out to be important for the properties of  pseudodifferential operators with symbols of Gevrey, analytic and ultra-analytic regularity, see \cite{CNR} for details.
As a corrolary of Theorem \ref{glavnaTeoremaSTFT} we obtain a Paley-Wiener type theorem for element of modulation spaces related to the extended Gevrey classes.
This result extends \cite[Theorem 3.1]{PTT-04} in the sense that the condition on compact support is replaced by appropriate decay property given by a modulation space norm, when the  Fourier-Laplace transform is replaced by the STFT.

In Section \ref{secWFstft} we recall the notion of wave front sets related to extended Gevrey regularity. We prove that such wave front sets can be characterized by the decay properties of the STFT of a distribution with respect to a suitably chosen window function, Theorem \ref{NezavisnostWFThm}. As a consequence we derive a result on local extended Gevrey regularity, Theorem
\ref{NezavisnostWFThm}.

Our results are proved for the so-called Roumieu case, and we note that proofs for the Beurling case  are similar and therefore omitted.

\subsection{Basic notions and notation} \label{notions}
We denote by ${\bf N}$, $\Z$, ${\bf R}$, ${\bf C}$ the sets of nonnegative
integers, positive integers, real numbers and complex numbers, respectively. For $x \in \Ren$
we put  $\langle x \rangle=(1+|x|^2)^{1/2}$.
The integer parts (the floor and the ceiling functions) of  $x\in {\bf R}_+$
are denoted by $\lf x \rf:=\max\{m\in \N\,:\,m\leq x\}$ and $\lceil x \rceil:= \min\{m\in \N\,:\,m\geq x\}$. For a multi-index
$\alpha=(\alpha_1,\dots,\alpha_d)\in {\bf N}^d$ we write
$\partial^{\alpha}=\partial^{\alpha_1}\dots\partial^{\alpha_d}$, $\dss D^{\alpha}= (-i )^{|\alpha|}\partial^\alpha$, and
$|\alpha|=|\alpha_1|+\dots |\alpha_d|$. Open ball of radius $r>0$ centered at $x_0$ is denoted by $B_r(x_0)$.
As usual, $C^\infty  (\mR^d) $ is the space of smooth functions, the Schwartz space of rapidly decreasing functions is denoted by $ \mathcal{S} (\mR^d) $, and $ \mathcal{S}' (\mR^d) $  denotes its dual space of tempered distributions. Lebesgue spaces over an open set $ \Omega \subseteq \Ren$ are denoted by $ L^p (\Omega) $, $ 1\leq p <\infty$, and the norm of $f \in L^p (\Omega) $ is denoted by $\| f\|_{L^p}.$

\par

The Fourier transform is normalized to be $${\hat   {f}}(\o)=\Fur f(\o)=\int f(t)e^{-2\pi i t\o}dt.$$
We use the brackets $\la f,g\ra$ to denote the extension of the inner product $\la f,g\ra=\int f(t){\overline
{g(t)}}dt$ on $L^2(\Ren)$ to the dual pairing between a test function space $ \mathcal A $ and its dual $ {\mathcal A}' $:
$ \langle \cdot, \cdot \rangle  = $ $ _{{\mathcal A}'}\langle \cdot, \overline{\cdot} \rangle _{\mathcal A}.$

\par

Translation and modulation operators, $T$ and $M$ respectively, when acting on $ f \in L^2 (\Ren)$
are defined by
$$
T_x f(\cdot) = f(\cdot - x) \;\;\; \mbox{ and } \;\;\;
 M_x f(\cdot) = e^{2\pi i x \cdot} f(\cdot), \;\;\; x \in \Ren.
$$
Then for $f,g \in L^2 (\Ren)$ the following relations hold:
$$
M_y T_x  = e^{2\pi i x \cdot y } T_x M_y, \;\;
 (T_x f)\hat{} = M_{-x} \hat f, \;\;
  (M_x f)\hat{} = T_{x} \hat f, \;\;\;
  x,y \in \Ren.
$$
These operators are extended to other spaces of functions and distributions in a natural way.

Throughout the paper, $A\lesssim B$
denotes $A\leq c B$ for a suitable constant $c>0$, whereas $A
\asymp B$ means that $c\inv A \leq B \leq c A$ for some $c\geq 1$. The
symbol $B_1 \hookrightarrow B_2$ denotes the continuous and dense embedding of
the topological vector space $B_1$ into $B_2$.

\section{Preliminaries} \label{sec01}

In this section we collect the main tools and auxiliary results which will be used in the sequel.
More precisely, we introduce the test function spaces related to the sequences of the form
$M_p^{\tau,\s}=p^{\tau p^{\s}}$, $p\in \Z$, for a given $\tau>0$ and $\s>1$. Notice that, when $\tau >1$ and $\s = 1$
$M_p^{\tau,1}=p^{\tau p}$, $p\in \Z$  is (equivalent to) the Gevrey sequence. Then we discuss associated functions to such sequences,
which are the main tool of our analysis. To describe precise asymptotic behavior of those associated functions at infinity appears to be
a nontrivial problem, which
can be resolved by the use of Lambert's $W$ functions. Finally, we recall the definition and some elementary properties of the STFT and
modulation spaces defined by mixed weighted Lebesgue norm conditions on the STFT.

\subsection{Extended Gevrey regularity}\label{secTest}

In this section we introduce extended Gevrey classes and discuss their basic properties.
We employ Komatsu's approach \cite{Komatsuultra1} to spaces of
ultradifferentiable functions, and consider defining sequences of the form
$M_p^{\t,\s}=p^{\t p^{\s}}$, $p\in \N $, depending on parameters $\t>0$ and  $\s>1$, \cite{PTT-02}.

Essential properties of the defining sequences are listed in the following lemma. We refer to \cite{PTT-01} for the proof.
In the general theory of ultradistributions (see \cite{Komatsuultra1}
different properties of defining sequences give rise to particular structural properties of the
corresponding spaces of
ultradifferentiable functions, see \cite{RS} for a detalied survey.

\par

\begin{lema}
\label{osobineM_p_s}
Let $\tau>0$, $\s>1$ and $M_p^{\tau,\s}=p^{\tau p^{\s}}$, $p\in \Z$, $M_0^{\tau,\s}=1$.
Then there exists an  increasing sequence of positive numbers $C_q$, $q\in \N$, and a constant $C>0$ such that:
\vspace{2mm}\\
\vspace{1mm}
$(M.1)$ $(M_p^{\t,\s})^2\leq M_{p-1}^{\t,\s}M_{p+1}^{\t,\s}$, $p\in \Z$\\
\vspace{1mm}
$\overline{(M.2)}$ $M_{p+q}^{\t,\s}\leq C^{p^{\s}+q^{\s}}M_p^{\t 2^{\s-1},\s}M_q^{\t 2^{\s-1},\s}$, $p,q\in \N$,\\
\vspace{1mm}
$\overline{(M.2)'}$ $M_{p+q}^{\t,\s}\leq C_q^{p^{\s}}M_p^{\t,\s}$, $p,q\in \N$,\\
\vspace{1mm}
$(M.3)'$
$ \displaystyle
\suml_{p=1}^{\infty}\frac{M_{p-1}^{\t,\s}}{M_p^{\t,\s}} <\infty.
$ Moreover, $\dss \frac{M_{p-1}^{\t,\s}}{M_p^{\t,\s}}\leq \frac{1}{(2p)^{\tau (p-1)^{\s-1}}}$, $p\in \N$.
\end{lema}

Let $\tau,h>0$, $\s>1$ and let $K\subset \subset \Ren$ be a regular compact set.  By ${\E}_{\t, {\s},h}(K)$
we denote
the Banach space of  functions $\phi \in  C^{\infty}(K)$ such that
\begin{equation} \label{Norma}
\| \phi \|_{{\E}_{\t, {\s},h}(K)}=\sup_{\alpha \in \N^d}\sup_{x\in K}
\frac{|\partial^{\alpha} \phi (x)|}{h^{|\alpha|^{\s}}  M_{|\alpha|} ^{\t,\s} }<\infty.\,
\end{equation}

The set of functions $ \phi \in {\E}_{\t,\s,h}(K)$ whose support is contained in $K$ is denoted by  ${\D}^K_{\t, \s,h}$.

Let $U$ be an open set $\Ren$ and $ K \subset \subset U$. We define families of spaces
by introducing the following projective and inductive limit topologies:
\begin{equation*}
\label{NewClassesInd} {\E}_{\{\t,
\s\}}(U)=\varprojlim_{K\subset\subset U}\varinjlim_{h\to
\infty}{\E}_{\t, {\s},h}(K),
\end{equation*}
\begin{equation*}
\label{NewClassesProj} {\E}_{(\t,
\s)}(U)=\varprojlim_{K\subset\subset U}\varprojlim_{h\to 0}{\E}_{\t,
{\s},h}(K),
\end{equation*}
\begin{equation*}
\label{NewClassesInd2} {\D}_{\{\t,
\s\}}(U)=\varinjlim_{K\subset\subset U} {\D}^K_{\{\t, \s\}}
=\varinjlim_{K\subset\subset U} (\varinjlim_{h\to\infty}{\D}^K_{\t,
\s,h})\,,
\end{equation*}
\begin{equation*}
\label{NewClassesProj2} {\D}_{(\t,
\s)}(U)=\varinjlim_{K\subset\subset U} {\D}^K_{(\t, \s)}
=\varinjlim_{K\subset\subset U} (\varprojlim_{h\to 0}{\D}^K_{\t,
\s,h}).
\end{equation*}
We will use abbreviated notation $ \t,\s $ for
$\{\t,\s\}$ (the Roumieu case) or $(\t,\s)$  (the Beurling case) .
The spaces ${\E}_{\t, \s}(U)$, ${\D}^K_{\t, \s}$ and ${\D}_{\t, \s}(U)$
are nuclear, cf. \cite{PTT-01}.
We refer to \cite{PTT-01, PTT-02, PTT-03, TT0, TT, FT} for other properties of those spaces.

\par

\begin{rem} \label{Gevrey-ext-Gevrey}
If $ \t > 1 $ and $\s = 1$,
then $ {\E}_{\{\t, 1\}}(U)={\E}_{\{\t\}}(U)$  is the Gevrey class,
and $\D_{\{\t,1\}}(U)=\D_{\{\t\}}(U)$
is its subspace of compactly supported functions in $\E_{\{\t\}}(U)$.

In particular,
\begin{equation*}
\label{GevreyNewclass}
\varinjlim_{t\to\infty} \E_{\{t\}}(U)\hookrightarrow {\E}_{\tau, \s}(U)
\hookrightarrow  C^{\infty}(U), \;\;\;
\tau>0, \; \s>1,
\end{equation*}
so that  the regularity in ${\E}_{\tau, \s}(U)$ can be thought of as an extended Gevrey regularity.

If $0<\t\leq 1$, then $ {\E}_{\t, 1}(U)$ consists of quasianalytic functions. In particular, $\dss
\D_{\t,1}(U)=\{0\}$ when $0<\t\leq 1$, and ${\E}_{\{1, 1\}}(U)= {\E}_{\{1\}}(U)$ is the space of analytic functions on $U$.
\end{rem}

\par

The non-quasianalyticity condition $(M.3)'$ provides the existence of
partitions of unity in  $\E_{\{\t,\s\}}(U)$, i.e. for any given $\t>0$ and $\s>1$,
there exists a compactly supported function $\phi\in {\E_{\{\t,\s\}}}(U)$ such that $0\leq\phi\leq 1$ and $\int_{\Ren}\phi\,dx=1$, see
\cite{PTT-01} for a construction of a compactly supported  $\phi \in {\D}_{\{\t, \s\}}(U) \setminus {\D}_{\{t\}}(U)$,  $t>1$.

\par

Note that the additional exponent $\s$, which appears in the power of term $h$ in \eqref{Norma},
makes the definition of $\E_{\t,\s}(U)$ different from the definition of  Carleman classes, cf. \cite{HermanderKnjiga}.
This difference appears to be essential in many calculations, and in particular when dealing with the
operators of ``infinite order``, cf. \cite{PTT-02}.

\subsection{The associated function to the sequence $M^{\t,\s}_p=p^{\t p ^{\s}}$}\label{secAsocirana}
In this subsection we recall the definition and asymptotic proeprties of
\emph{extended associated function} to the sequence $M^{\t,\s}_p=p^{\t p^{\s}}$, $p\in\N$, $\t>0$, $\s>1$, cf. \cite{PTT-04}.
We also recall the Paley-Wiener theorem related to the extended Gevrey regularity.

\begin{de} \label{De:asocirana}
Let $\tau>0$, $\s>1$ and $M_p^{\tau,\s}=p^{\tau p^{\s}}$, $p\in \Z$, $M_0^{\tau,\s}=1$.
The extended associated function related to the sequence $M_p^{\t,\s}$, is given by
$$
\dss T_{\t,\s,h}(k)=\sup_{p\in \N}\ln_+\frac{h^{p^{\s}}k^{p}}{M_p^{\t,\s}}, \;\;\; h,k>0,
$$
where $\dss \ln_+ A=\max\{0, \ln A\}$, for $A>0$.
\end{de}

Obviously $T_{\t,\s,h}(k)$,  $\t,h>0$, $\s>1$,  is positive for sufficiently large  $k>0$.

In fact, for any  sequence of positive numbers  $M_p$, $p\in \N$, such that $M^{1/p}_p$ is bounded from below and $M_0 = 1$, its \emph{associated function}
is defined to be
$$T(k)=\sup_{p\in\N}\ln \frac{k^{p}}{M_p},\quad k>0.$$
Therefore, for $\t>0$ and $\s=1$, $T_{\t,1,h}(k):=T_{\t}(hk)$ is the associated function to the Gevrey sequence $p^{\t p}$, $p\in\N$ (we may assume $h=1$ without loosing generality).

It is well known (cf. \cite{GelfandShilov, Rodino}) that
\be
\label{NejednakostGevrey}
 A k^{1/\t}-B \leq T_{\t}(k)\leq A k^{1/\t}, \quad k>0,
\ee for suitable $A,B>0$. In particular, the growth of $\dss e^{T_{\t}(k)}$ for $\t>1$ is subexponential.

Moreover, for any $t,\t>0$ and $\s>1$, by \cite[Lemma 2.3]{PTT-04} it follows that there is a constant $C>0$ such that
$$T_{\t,\s,1}(k)< C k^{1/t},\quad k>0.$$
Therefore the function $e^{T_{\t,\s,h}(k)}$ has a less rapid  growth at infinity  than any subexponential function.

The precise asymptotic behavior of $T_{\t,\s,h}(k)$ at infinity is a challenging problem. We use an auxiliary special function to resolve that problem.

\par

The Lambert $W$ function is defined as the inverse function of $z e^{z}$, $z\in {\bf C}$, wherefrom the following property holds:
$$
\dss x=W(x)e^{W(x)}, \quad x\geq 0.
$$
We denote its principal (real) branch by $W(x)$, $x\geq 0$ (see \cite{LambF}).
It is a continuous,  increasing and concave function on $[0,\infty)$, $W(0)=0$, $W(e)=1$, and $W(x)>0$, $x>0$.

It can be shown that $W$ can be represented in the form of the absolutely convergent series
$$
W(x)=\ln x-\ln (\ln x)+\sum_{k=0}^{\infty}\sum_{m=1}^{\infty}c_{km}\frac{(\ln(\ln x))^m}{(\ln x)^{k+m}},\quad x\geq x_0>e,
$$
with suitable constants $c_{km}$ and  $x_0 $, wherefrom  the following  estimates hold:

\be
\label{sharpestimateLambert}
\ln x -\ln(\ln x)\leq W(x)\leq \ln x-\frac{1}{2}\ln (\ln x), \quad x\geq e.
\ee
The equality in \eqref{sharpestimateLambert} holds if and only if $x=e$.
We refer to \cite{HoHa, LambF} for more details about the Lambert $W$ function.

\par

\begin{te} (\cite{PTT-04})
\label{propozicija}
Let there be given $\t, h>0$, $\s>1$ and let  $C_{\t,\s,h} = h^{-\frac{\s-1}{\t}}e^{\frac{\s-1}{\s}}\frac{\s-1}{\t \s}$. Then
\begin{multline}
\label{nejednakostzaTeoremu1}
\exp\Big\{ ( 2^{\s-1 }\t)^{-\frac{1}{\s-1}}{\Big(\frac{\s-1}{ \s}\Big)^{\frac{\s}{\s-1}} \, {W^{-\frac{1}{\s-1}}(C_{\t,\s,h}\ln k)}\,{\ln}^{\frac{\s}{\s-1}}k }\Big\}
\lesssim e^{T_{\t,\s,h}(k)}\\
\lesssim
exp\Big\{{\Big(\frac{\s-1}{\t \s}\Big)^{\frac{1}{\s-1}}\, {W^{-\frac{1}{\s-1}}(C_{\t,\s,h} \ln k)}\,{\ln}^{\frac{\s}{\s-1}}k }\Big\}, \quad k>e.
\end{multline}
If, moreover $1<\s<2$, then we have the precise asymptotic formula
$$
\label{nejednakostzaTeoremu2}
e^{T_{\t,\s,h}(k)} \asymp
\exp\Big\{{\Big(\frac{\s-1}{\t \s}\Big)^{\frac{1}{\s-1}}\, {W^{-\frac{1}{\s-1}}(C_{\t,\s,h}\ln k)}\,{\ln}^{\frac{\s}{\s-1}}k }\Big\}, \quad k>e.
$$
The hidden constants in \eqref{nejednakostzaTeoremu1} and \eqref{nejednakostzaTeoremu2} depend on $\t,\s$ and $h$.
\end{te}

\begin{rem}
\label{RemarkAsimtotska}
Note that, in the view of \eqref{sharpestimateLambert} we have
\begin{multline}
\label{asimptotskaocena}
 {W^{-\frac{1}{\s-1}}(C \ln k)}\,{\ln}^{\frac{\s}{\s-1}}k
 \asymp
 \frac{\ln^{\frac{\s}{\s-1}} k}{\ln^{\frac{1}{\s-1} }(C \ln k)} \asymp \frac{\ln^{\frac{\s}{\s-1}} k}{\ln^{\frac{1}{\s-1} }(\ln k)},\quad k\to \infty,
\end{multline}
for any given $\s > 1$, and the last behavior follows from  $\dss \ln(C\ln k) \asymp \ln(\ln k)$, $k\to \infty$, for any given  $C>0$.

\par

Since $\lim_{k\to\infty} ( \ln k)^{1/(\s-1)} (\ln (C \ln k) )^{-1/(\s-1)} =\infty,$
for every $C>0$,
\eqref{asimptotskaocena} implies that for every $M>0$ there exists $B>0$ (depending on $h$ and $M$) such that
$$
{W^{-\frac{1}{\s-1}}(C \ln k))}\,{\ln}^{\frac{\s}{\s-1}}k> M\ln k,\quad k> B.
$$
\end{rem}

Next we recall the Paley-Wiener theorem for $\D^K_{\t,\s}$  when $1<\s<2$. For the proof we refer to \cite{PTT-05}, and a more general case  when $\s \geq 2$ is proved in \cite{PTT-04}.

\begin{theorem}
\label{PosledicaPaley}
Let $\t>0$, $1<\s<2$, $U$ be open set in ${\mathbf R}^d$ and $K\subset\subset U$. If $\varphi\in \D_{\{\t,\s\}}^{K}$ (resp. $\varphi\in \D_{(\t,\s)}^{K}$) then its Fourier-Laplace transform is an entire function and there exists constants $A,B>0$ (resp. for every $B>0$ there exists $A>0$) such that

\begin{multline}
\label{ocenaTeorema0}
|\widehat\varphi(\eta)|\leq  A \exp\Big\{-\Big(\frac{\s-1}{\t \s}\Big)^{\frac{1}{\s-1}}{W^{-\frac{1}{\s-1}}\Big(B \ln (e+|\eta|)\Big)}{{\ln}}^{\frac{\s}{\s-1}}(e+|\eta|)+ H_K(\eta) \Big\}\\ \quad h>0,\, \eta \in {\mathbf C}^d,
\end{multline} where $\dss H_K(\eta)=\sup_{y\in K} {\rm Im}(y\cdot\eta)$.

Conversely, if there exists $A,B>0$ (resp. for every $B>0$ there exists $A>0$) such that an entire function $\widehat{\varphi}(\eta)$ satisfies \eqref{ocenaTeorema0} then $\widehat{\varphi}(\eta)$ is the Fourier-Laplace transform of  $\varphi\in \D_{\{2^{\s-1}\t,\s\}}^{K}$ (resp.  $\D_{(2^{\s-1}\t,\s)}^{K}$).
 \end{theorem}

The following corollary is an immediate consequence of Theorem \ref{PosledicaPaley} and \eqref{asimptotskaocena}.

\begin{cor}
\label{PoslednjaTeorema}
Let $1<\s<2$, $U$ be open set in ${\mathbf R}^d$ and $K\subset\subset U$. Then the entire function $\widehat{\varphi}(\eta)$, $\eta\in {\mathbf C}^d$, is the Fourier-Laplace transform of  $$\varphi\in \varinjlim_{\t\to\infty}\D^K_{\t,\s}\, (\rm{resp}.\,\, \varphi\in \varprojlim_{\t\to 0}\D^K_{\t,\s})$$ if and only if there exist constant $A,B>0$ (resp. for every $B>0$ there exists $A>0$) such that
$$
|{\widehat\varphi}(\eta)|\leq A \exp\left \{-B\frac{\ln^{\frac{\s}{\s-1}}(e+|\eta|)}{\ln^{\frac{1}{\s-1}}(\ln (e+|\eta|))}+H_K(\eta)\right \}, \eta \in {\mathbf C}^d,
$$
where $\dss H_K(\eta)=\sup_{x\in K} {\rm Im}(x\cdot\eta)$.
\end{cor}

\subsection{Modulation Spaces} \label{Modulation Spaces}

The modulation spaces were initially (and systematically) introduced in \cite{F1}. See also\cite[Ch.~11-13]{Gro} and
the original literature quoted there for various properties and applications of the so called {\em standard} modulation spaces.
It is usually sufficient to observe weighted modulation spaces
with weights which may grow  at most polynomially at infinity.
However, for the study  of  ultra-distributions a more general approach which includes  weights of exponential or even superexponential growth is needed,
cf. \cite{CPRT1, Toft-2017}. We refer to \cite{FG-Atom1, FG-Atom2}
for related but even more general constructions, based on the general theory of coorbit spaces.

\par

For our purposes it is sufficient to consider weights of exponential growth. Therefore we begin with
the Gelfand-Shilov space of analytic functions   $  {\mathcal S} ^{(1)} (\mathbb{R}^d ) $ given by
$$
f \in  {\mathcal S}^{( 1)} (\mathbb{R}^d) \Longleftrightarrow
\sup_{x\in \mathbb{R}^d } |f(x)  e^{h\cdot |x|}| < \infty \;
\; \text{and} \;
\sup_{\omega \in \mathbb{R}^d } | \hat f (\omega)  e^{h\cdot |\omega|} | < \infty,
$$
for every $ h > 0.$ Any $ f \in  {\mathcal S}^{( 1)} (\mathbb{R}^d) $ can be extended to a
holomorphic function $f(x+iy)$ in the strip
$ \{ x+iy \in \mC ^d \; : \; |y| < T \} $ some $ T>0$, \cite{GelfandShilov, NR}.
The dual space of  $  {\mathcal S} ^{(1)} (\mathbb{R}^d ) $ will be denoted by
$ {\mathcal S} ^{(1)'}   (\mathbb{R}^d ). $
In fact,  $  {\mathcal S} ^{(1)} (\mathbb{R}^d ) $ is isomorphic to the Sato test function space for the space of Fourier hyperfunctions $ {\mathcal S} ^{(1)'}   (\mathbb{R}^d ), $ see \cite{CCK1994}.


Let there be given $f, g \in L^2 (\mathbb{R}^d).$ The short-time Fourier transform (STFT) of
$ f  $ with respect to the window $g$  is given by
\be \label{GT}
V_g f (x,\omega) = \int e^{-2\pi i t \omega} f(t)  \overline{ g(t-x)}dt, \;\;\; x,\omega \in \mathbb{R}^d.
\ee
It restricts to a mapping from $  {\mathcal S} ^{(1)} (\mathbb{R}^d ) \times   {\mathcal S} ^{(1)} (\mathbb{R}^d ) $ to
$  {\mathcal S} ^{(1)} (\mathbb{R}^{2d} ) $, which is proved in the next Lemma.

\begin{lemma} \label{Lm:STFTonS1}
Let  $f,g \in \cS^{(1)} (\Ren )$, and let the short-time Fourier transform (STFT) of
$ f $ with respect to $g$ be given by \eqref{GT}. Then
$V_g f (x,\omega) \in {\mathcal S} ^{(1)} (\mathbb{R}^{2d} ) $, that is
$ |V_gf(x,\o)| < C e^{-s \|\phas\| }$, $x,\omega \in \mathbb{R}^d$, for every $s > 0.$
\end{lemma}

\begin{proof} The proof is standard, see e.g. \cite{Gro} for the proof in the context of  $ {\mathcal S}  (\mathbb{R}^{d} ) $.
We use the arguments based on the structure of $ \cS^{(1)} (\Ren )$ as follows.
Let $ f \otimes g $ be the tensor product $ f \otimes g (x,t)  = f(x)\cdot g(t),$ let
$\mathcal{T} $ denote the asymmetric coordinate transform $ \mathcal{T} F(x,t) = F(t,t-x)$,
and let $ \mathcal{F}_2 $ be the partial Fourier transform
$$
\mathcal{F}_2 F (x,\omega) = \int_{\Ren}  F(x,t) e^{-2\pi i t \omega} dt, \quad x,\omega \in \mathbb{R}^d,
$$
of a function $F$ on $\Renn$. Then
$$
V_g f (x,\o)  = \mathcal{F}_2 \mathcal{T} (f \otimes g )(x,\o), \quad (x,\o)\in \Renn.
$$
Since $ \displaystyle \cS ^{(1)}(\Renn)\cong \cS ^{(1)}(\Ren) \hat{\otimes} \cS ^{(1)}(\Ren) $ (see e.g. \cite{Teof2015} for the kernel theorem in Gelfand-Shilov spaces)
and since $ \displaystyle \cS ^{(1)}(\Renn)$ is invariant under the action of $\mathcal{T} $ and $ \mathcal{F}_2 $, we conclude that
$ |V_g f(x,\o)| < C e^{-s \|\phas\| } $, $ x,\omega \in \mathbb{R}^d,$ for every $s > 0.$
\end{proof}

\emph{Weight Functions.}  In the sequel $v$ will always be a
continuous, positive,  even, submultiplicative   function
(submultiplicative weight), i.e., $v(0)=1$, $v(z) =
v(-z)$, and $ v(z_1+z_2)\leq v(z_1)v(z_2)$, for all $z,
z_1,z_2\in\Renn.$ Moreover, $v$ is assumed to be even in each group of coordinates,  that is, $ v (x, \o) = v (-\o, x)= v(-x,\o), $ for any $\phas\in\Renn$. Submultipliciativity implies that $v(z)$ is \emph{dominated} by an exponential function, i.e.
\begin{equation} \label{weight}
 \exists\, C, k>0 \quad \mbox{such\, that}\quad  v(z) \leq C e^{k \|z\|},\quad z\in \Renn,
\end{equation}
and   $\|z\|$ is  the Euclidean norm of $z\in \Renn$.
For example, every weight of the form
$$
v(z) =   e^{s\|z\|^b} (1+\|z\|)^a \log ^r(e+\|z\|)
$$
 for parameters $a,r,s\geq 0$, $0\leq b \leq 1$ satisfies the
above conditions.

\par

Associated to every submultiplicative weight we consider the class of
so-called  {\it
  v-moderate} weights $\cM _v$. A  positive, even
weight function  $m$ on $\Renn$ belongs to $\cM _v$ if it  satisfies
the condition
$$
 m(z_1+z_2)\leq Cv(z_1)m(z_2)  \quad   \forall z_1,z_2\in\Renn \, .
$$
 We note that this definition implies that
$\frac{1}{v} \lesssim m \lesssim v $,  $m \neq 0$ everywhere, and that
$1/m \in \cM _v$.

The widest class of weights allowing to
define \modsp s is the weight class $\cN$. A weight function  $m$
on $\Renn$ belongs to $\cN$  if it is a continuous, positive
function such that
$$
m(z)=o(e^{c z^2}),\,\quad\mbox{for}\,\,|z|\rightarrow\infty,\quad
\forall c>0,
$$
with $z\in\Renn$. For instance, every function $m(z)=e^{s |z|^b}$,
with $s>0$ and $0\leq b<2$,  is in $\cN$. Thus, the weight $m$ may
grow faster than exponentially at infinity. For example, the choice $ m \in \cN \setminus \cM _v$
is related to the spaces of  quasianalytic functions, \cite{CPRT2}.
We notice that  there
is a limit in enlarging the weight class for \modsp s, imposed by
Hardy's theorem:
if $m(z)\geq C e^{c z^2}$, for some $c>\pi/2$, then the corresponding \modsp s  are trivial \cite{GZ01}.
We refer to \cite{Gro07} for a
survey on the most important types of weights commonly used in time-frequency analysis.

\begin{definition}\label{defmodnorm}
Let $v$ be  a submultiplicative weight $v$, $m\in \cM _v$,   and let  $g$ be a non-zero \emph{window} function in $\cS ^{(1)}(\Ren)$. For
$1\leq p,q\leq \infty$ the {\it  modulation space} $M^{p,q}_m(\Ren)$ consists of all
$f\in {\mathcal S} ^{(1)'}   (\Ren) $
such that $V_gf\in L^{p,q}_m (\Ren)$
(weighted mixed-norm spaces). The norm on $M^{p,q}_m (\Ren)$ is
$$
\|f\|_{M^{p,q}_m}=\|V_{g}f\|_{L^{p,q}_m}=\left(\int_{\Ren}
  \left(\int_{\Ren}|V_{g} f(x,\o)|^pm(x,\o)^p\,
    dx\right)^{q/p}d\o\right)^{1/q}
$$
(with obvious changes if either $p=\infty$ or $q=\infty$). If
$p,q< \infty $, the \modsp\ $\Mmpq (\Ren) $ is the norm completion of
$\cS ^{(1)}(\Ren)$ in the $\Mmpq $-norm. If $p=\infty $ or
$q=\infty$, then $\Mmpq (\Ren) $ is the completion of $\cS ^{(1)} (\Ren)$
in the weak$^*$ topology.
\end{definition}

Note that for $f,g \in \cS^{(1)} (\Ren )$
the above integral is convergent so that $ \cS ^{(1)} (\Ren) \subset M^{p,q}_m  (\Ren).$
Namely, in  view of \eqref{weight}, for a given $ m\in\cM _v$ there exist $l>0$ such that
$ m (x,\o) \leq C e^{l \|\phas\|}$ and therefore
\begin{eqnarray*}
&& \left|\int_{\Ren}
  \left ( \int_{\Ren}|V_gf(x,\o)|^p m(x,\o)^p\,
    dx\right)^{q/p}d\o\right|\\
   && \quad\quad\quad\quad
     \leq
    C \left|\int_{\Ren}
  \left( \int_{\Ren}|V_gf(x,\o)|^p e^{l p\|\phas\|}\,
    dx\right) ^{q/p}  d\o\right| < \infty,
\end{eqnarray*}
since by Lemma \ref{Lm:STFTonS1} it follows that $ |V_gf(x,\o)| < C e^{-s \|\phas\| } $ for every $s > 0.$

\par

If $p=q$, we write $M^p_m$ instead of $M^{p,p}_m$, and if $m(z)\equiv 1$ on $\Renn$, then we write $M^{p,q}$ and $M^p$ for $M^{p,q}_m$ and $M^{p,p}_m$, and so on.

\par
In the next proposition we show that  $\Mmpq (\Ren )$ are  Banach spaces
whose definition is independent of the choice of the window
$g \in M^1_{v}(\Ren ) \setminus \{ 0\}$.
In order to do so, we need the adjoint of the short-time Fourier transform.

For a given window $ g \in  \mathcal{S} ^{(1)} (\Ren )$ and a
function $ F (x,\xi) \in L^{p,q} _m (\R^{2d})$ we define $ V^* _g F $ by
$$
\langle V^* _g F, f \rangle := \langle F, V_g f \rangle,
$$
whenever the duality is well defined.

Then \cite[Proposition 11.3.2]{Gro} (see also \cite{CPRT1}) can be rewritten as follows.

\begin{proposition} \label{emjedanve}
Fix $m \in  \cM _v$ and $ g, \psi \in  \mathcal{S} ^{(1)},$ with $\langle g, \psi \rangle\not= 0$. Then
\begin{enumerate}
\item $ V^* _g :  L^{p,q} _m (\R^{2d}) \rightarrow  M^{p,q} _m (\R^{d}), $ and
$$
\|  V^* _g F  \|_{ M^{p,q} _m } \leq C \| V_\psi g \|_{ L^{1} _v }   \| F \|_{L^{p,q} _m}.
$$
\item The inversion formula holds: $ I_{ M^{p,q} _m }  = \langle g, \psi \rangle^{-1}
  V^* _g   V _\psi,$ where  $ I_{ M^{p,q} _m } $ stands for the identity operator.
\item   $\Mmpq (\Ren )$ are  Banach spaces
whose definition is independent on the choice of
$ g \in \mathcal{S} ^{(1)} \setminus \{ 0 \} $.
\item The space of admissible windows can be extended from $ {\mathcal S} ^{(1)}(\Ren ) $ to $M^1 _v(\Ren ) .$
\end{enumerate}
\end{proposition}

When $m$ is a polynomial weight of the form $m (x, \omega) = \langle  x \rangle ^t
\langle \omega \rangle ^s$ we will use the notation
$M^{p,q}_{s,t}(\mathbb{R}^d)$ for the modulation spaces which consists of all
$f\in ( \mathcal{S}^{(1)} ) '(\mathbb{R}^d)$ such that
$$
\| f \|_{M^{p,q}_{s,t}} \equiv \left  ( \int _{\mathbb{R}^d} \left ( \int _{\mathbb{R}^d}
|V_\phi f(x,\omega )\langle  x \rangle ^t
\langle \omega \rangle ^s|^p\, dx  \right )^{q/p}d\omega  \right )^{1/q}<\infty
$$
(with obvious interpretation of the integrals when $p=\infty$ or $q=\infty$).

\par

The following theorem lists some basic properties of modulation spaces.
We refer to \cite{F1, Gro, PT1, T3, Toft-2012} for its proof.

\begin{theorem} \label{modproerties}
Let $p,q,p_j,q_j\in [1,\infty ]$ and $s,t,s_j,t_j\in \mathbb{R}$, $j=1,2$. Then:
\begin{enumerate}
\item $M^{p,q}_{s,t}(\mathbb{R}^d)$ are Banach spaces, independent of the choice of
$\phi \in \mathcal{S}(\mathbb{R}^d) \setminus 0$;

\item if  $p_1\le p_2$, $q_1\le q_2$, $s_2\le s_1$ and
$t_2\le t_1$, then
$$
\mathcal{S}(\mathbb{R}^d)\subseteq M^{p_1,q_1}_{s_1,t_1}(\mathbb{R}^d)
\subseteq M^{p_2,q_2}_{s_2,t_2}(\mathbb{R}^d)\subseteq
\mathcal{ S}'(\mathbb{R}^d);
$$

\item $ \displaystyle
\cap _{s,t} M^{p,q}_{s,t}(\mathbb{R}^d)=\mathcal{ S}(\mathbb{R}^d),
\quad
\cup _{s,t}M^{p,q}_{s,t}(\mathbb{R}^d)=\mathcal{ S}'(\mathbb{R}^d); $
\item Let $ 1\leq p, q \leq  \infty, $ and let $   w_s (z)= w_s\phas = e^{s\|\phas\|},$
$   z=(x,\o)\in\Renn $.
Then
$$
{\mathcal S}^{(1)} (\mathbb{R}^d) =  \bigcap _{s \geq  0} M _{w_{s}} ^{p,q} (\mathbb{R}^d) = \bigcap _{m  \in \cap \cM _{w_s}} M _{m} ^{p,q} (\mathbb{R}^d),
$$
$$
{\mathcal S}^{(1)'} (\mathbb{R}^d) = \bigcup _{s \geq 0} M _{1/w_{s}} ^{p,q} (\mathbb{R}^d) =  \bigcup _{m  \in \cap \cM _{w_s}} M _{1/m} ^{p,q} (\mathbb{R}^d);
$$
\item For   $p,q\in [1,\infty )$, the dual of $ M^{p,q}_{s,t}(\mathbb{R}^d)$ is
$ M^{p',q'}_{-s,-t}(\mathbb{R}^d),$ where $ \frac{1}{p} +  \frac{1}{p'} $ $ =
 \frac{1}{q} +  \frac{1}{q'} $ $ =1.$
\end{enumerate}
\end{theorem}

Modulation spaces include the following well-know function spaces:
\begin{enumerate}
\item $ M^2 (\mathbb{R}^d) = L^2 (\mathbb{R}^d),$  and $ M^2 _{t,0}(\mathbb{R}^d) = L^2 _t (\mathbb{R}^d);$
\item The Feichtinger algebra: $ M^1 (\mathbb{R}^d) = S_0 (\mathbb{R}^d);$
\item Sobolev spaces: $ M^2 _{0,s}(\mathbb{R}^d) = H^2 _s (\mathbb{R}^d) = \{ f \, | \,
\hat f (\omega) \langle \omega \rangle ^s \in  L^2 (\mathbb{R}^d)\};$
\item Shubin spaces: $ M^2 _{s}(\mathbb{R}^d) = L^2 _s (\mathbb{R}^d) \cap H^2 _s (\mathbb{R}^d) = Q_s (\mathbb{R}^d),$
cf. \cite{Shubin91}.
\end{enumerate}

\section{Decay properties of the STFT}\label{secGlavniRez}

In this section we characterize certain regularity properties related to the classes $\E_{\t,\s}$, $\t>0$, $\s\geq1$, by the rate of decay of the STFT. In particular, we
extend \cite[Theorem 3.1]{CNR}, which is formulated in terms of Gevrey sequences and the corresponding spaces of test functions.

\par
In the proof of Theorem \ref{glavnaTeoremaSTFT} in several occasions we will use the following simple inequalities:
\be
\label{simpleInequality1}
|\alpha|^{\s}+|\beta|^{\s}\leq|\alpha+\beta|^{\s}\leq 2^{\s-1}(|\alpha|^{\s}+|\beta|^{\s}), \quad \alpha,\beta\in \N^d,\,\s>1,
\ee
and
\be
\label{simpleInequality2}
\dss |(1/\sqrt{d})\xi|^{|\alpha|}\leq |\xi^{\alpha}|\leq |\xi|^{|\alpha|},\quad \alpha\in \N^d,\, \xi \in \Ren.
\ee

\begin{te}
\label{glavnaTeoremaSTFT}
Let $\t>0$, $\s\geq 1$, let $v$ be  a submultiplicative weight,  $m\in \cM _v$,
and let $g\in M^{1,1} _{v\otimes 1} (\Ren) \backslash\{0\}$ such that for some $ C_g>0$,
\be
\label{uslovTeoreme}
\|\partial^{\alpha} g\|_{L^1 _v (\Ren)}
\lesssim C_g^{|\alpha|^{\s}}|\alpha|^{\t |\alpha|^{\s}},\quad \alpha \in {\bf N}^d.
\ee
For a smooth function $f$ the following conditions are equivalent:
\begin{itemize}
\item[i)] There exists a constant $C_f>0$ such that
\be
\label{Eq:f-estimate}
\|\partial^{\alpha} f\|_{L^{\infty} (\Ren)}
\lesssim m(x) C_{f}^{|\alpha|^{\s}} |\alpha|^{\t |\alpha|^{\s}}, \,\, \alpha \in {\bf N}^d;
\ee
\item[ii)] There exists a constant $C_{f,g}>0$ such that
$$ |\xi|^{|\alpha|}|V_{g}f(x,\xi)| \lesssim   m(x) C_{f,g}^{|\alpha|^{\s}}|\alpha|^{\t |\alpha|^{\s}},\,\,  x,\xi\in \Ren, \alpha \in {\bf N}^d; $$
\item[iii)] There exists a constant $C>0$ such that
$$ |V_{g}f(x,\xi)|\lesssim   m(x) e^{-T_{\t,\s,C}(|\xi|)}, \,\, x,\xi\in \Ren. $$
\end{itemize}
\end{te}

\begin{proof} When $\s=1$ we obtain \cite[Theorem 3.1]{CNR}, where the function $ C |x|^{1/\tau} $ appears instead of $T_{\t,\s,C}(|\xi|)$.
However,  this makes no difference, since
from $T_{\t,1,C}(|\xi|):=T_{\t}( C |\xi|)$  (see also \eqref{NejednakostGevrey}) it follows that
iii) is equivalent to
$$
\dss |V_{g}f(x,\xi)|\lesssim  m(x) e^{-C|\xi|^{1/\t}}, \quad\quad x,\xi\in \Ren,
$$
for some  $C>0$.
Note also that due to \eqref{simpleInequality2} the condition {\em (ii)} on $|\xi^{\alpha} V_{g}f(x,\xi)| $ given by (38) in \cite{CNR}
is equivalent to ii).

\par

Let $\s>1.$ We follow the proof of  \cite[Theorem 3.1]{CNR}, with necessary modifications, since we consider a more general situation.

i) $ \Rightarrow$ ii) Since $V_g f(x,\xi)=\Fur (f T_x \overline{g})(\xi)$,
we can formally write
\begin{multline*}
\xi^{\alpha} V_g f(x,\xi)=\frac{1}{(2\pi i)^{|\alpha|}} \Fur f(\partial^{\alpha}(f T_x \overline{g}))(\xi) \\[1ex]
=\frac{1}{(2\pi i)^{|\alpha|}}\sum_{\beta\leq\alpha}{\alpha\choose\beta}\Fur
(\partial^{\alpha-\beta}f \partial^{\beta } (T_x \overline{g} ))(\xi), \quad
x,\xi\in \Ren
\end{multline*}
(where we used the Leibnitz formula), and the formalism can be justified as follows.
\begin{multline*}
|\xi^{\alpha} V_g f(x,\xi)|\lesssim \frac{1}{(2\pi)^{|\alpha|}}\sum_{\beta\leq\alpha}{\alpha\choose\beta}
\|\Fur (\partial^{\alpha-\beta}f T_x(\partial^{\beta } \overline{g}))\|_{L^{\infty}}\\
\lesssim  \frac{1}{(2\pi)^{|\alpha|}}\sum_{\beta\leq\alpha}{\alpha\choose\beta}
\|\partial^{\alpha-\beta}f T_x(\partial^{\beta }  \overline{g})\|_{L^{1}}.
\end{multline*}

Since $m$ is a positive $v-$moderate weight, from \eqref{uslovTeoreme},  \eqref{Eq:f-estimate}, and H\"older's inequality we obtain
\begin{multline*}
\|\partial^{\alpha-\beta}f T_x(\partial^{\beta }  \overline{g})\|_{L^{1}}
\leq
\|\partial^{\alpha-\beta}f \|_{L^{\infty} _{1/m}}
\|m(x) (\partial^{\beta } T_x \overline{g})\|_{L^{1}} \\
\lesssim
C_{f}^{|\alpha-\beta|^{\s}}|\alpha-\beta|^{\t |\alpha-\beta|^{\s}} m(x)
\|v(\cdot -x) \partial^{\beta }  \overline{g} (\cdot -x) \|_{L^{1}} \\
\lesssim
m(x) C_{f}^{|\alpha-\beta|^{\s}}|\alpha-\beta|^{\t |\alpha-\beta|^{\s}}
\cdot C_{g}^{|\beta|^{\s}}|\beta|^{\t |\beta|^{\s}} \\
\lesssim
m(x) \tilde C_{f,g}^{|\alpha|^{\s}} |\alpha|^{\t |\alpha|^{\s}},
\end{multline*}
where we used the fact that the sequence $M_p=p^{\t p^\s}$ satisfies
$$
(M.1)': \; M^{\t,\s}_{p-q} M^{\t,\s}_q\leq M^{\t,\s}_p, \quad q\leq p, \quad p,q \in \N, $$
which follows from $(M.1)$, see Lemma \ref{osobineM_p_s}, and also \cite{Komatsuultra1}.

Thus
\begin{multline*}
|\xi|^{|\alpha|} |V_g f(x,\xi)|\lesssim (\sqrt{d})^{|\alpha|} |\xi ^{\alpha} V_g f(x,\xi)| \\[1ex]
\lesssim \left ( \frac{\sqrt{d}}{2\pi} \right )^{|\alpha|} m(x)
\sum_{\beta\leq\alpha}{\alpha\choose\beta} \tilde  C_{f,g}^{|\alpha|^{\s}}|\alpha|^{\t |\alpha|^{\s}} \\[1ex]
\lesssim \left ( \frac{\sqrt{d}}{\pi} \right )^{|\alpha|}  m(x) |\alpha|^{\t |\alpha|^{\s}} \tilde C_{f,g}^{|\alpha|^{\s}}
= C_{f,g}^{|\alpha|^{\s}} |\alpha|^{\t |\alpha|^{\s}}, \quad x,\xi\in \Ren, \alpha \in {\bf N}^d,
\end{multline*}
where we used \eqref{simpleInequality2}, and ii) follows.

ii) $\Rightarrow$ i) Note that  \eqref{Eq:f-estimate} means that $f\in M^{\infty,1} _{m^{-1} \otimes 1} (\Ren)$. Hence we may use the inversion formula for STFT
(cf. Proposition 11.3.2. in \cite{Gro}), and since $f$ is a smooth function, we may assume that it holds everywhere. So we formally write
\begin{multline*}
\partial^{\alpha} f(t)=\frac{1}{\|g\|^2 _{L^2} }\int_{\mR^{2d}} V_g f(x,\xi)\partial^{\alpha} (M_{\xi} T_x  \overline{g})(t)\,dx d\xi\quad\\
= \frac{1}{\|g\|^2 _{L^2}} \sum_{\beta\leq \alpha}{\alpha\choose \beta}\int_{\mR^{2d}} V_g f(x,\xi)(2\pi i \xi)^{\beta} M_{\xi} T_x(\partial^{\alpha-\beta}  \overline{g})(t)\,dx d\xi,
\end{multline*}
$\alpha\in \N^d, t\in \Ren, $ where we used the Leibnitz formula.
The estimates below also justify the exchange of the order  of derivation  and integration.
Therefore,
\begin{multline*}
|\partial^{\alpha} f(t)|
\lesssim
\frac{1}{\|g\|^2 _{L^2} }
\sum_{\beta\leq \alpha}{\alpha\choose \beta} (2\pi )^{|\beta|}
\int_{\mR^{2d}} | V_g f(x,\xi) \xi ^{\beta} | |T_x(\partial^{\alpha-\beta}  \overline{g})(t)|\,dx d\xi, \\
\lesssim
\frac{1}{\|g\|^2 _{L^2} }
\sum_{\beta\leq \alpha}{\alpha\choose \beta} (2\pi )^{|\beta|} I_{\alpha,\beta} (t), \quad  t\in \Ren,
\end{multline*}
where we put
$$
I_{\alpha,\beta} (t) = \int_{\mR^{2d}} | V_g f(x,\xi) \xi^{\beta}|  |T_x(\partial^{\alpha-\beta}  \overline{g})(t)|\,dx d\xi,\quad  t\in \Ren.
$$
We note that ii) is equivalent with
$$
\langle \xi \rangle ^{|\alpha|}|V_{g}f(x,\xi)| \lesssim   m(x) C_{f,g}^{|\alpha|^{\s}}|\alpha|^{\t |\alpha|^{\s}},\,\,  x,\xi\in \Ren, \alpha \in {\bf N}^d,
$$
and estimate $ I_{\alpha,\beta} (t) $ as follows:
\begin{multline*}
I_{\alpha,\beta} (t) \leq \int_{\mR^{2d}} \langle \xi \rangle^{\beta} | V_g f(x,\xi)| \frac{\langle \xi \rangle^{d+1}}{\langle \xi \rangle^{d+1}} |\overline{g}^{(\alpha-\beta)}(t-x)|\,dx d\xi \\
\lesssim \int_{\mR^{2d}} | V_g f(x,\xi)| \frac{\langle \xi \rangle^{\beta +d+1}}{m(x)}
\frac{m(x)}{\langle \xi \rangle^{d+1}}
 |\overline{g}^{(\alpha-\beta)}(t-x)|\,dx d\xi \\
\lesssim C_{f,g} ^{|\beta + d+1|^\s} |\beta +d+1|^{\t |\beta +d+1|^\s}
 \int_{\mR^{2d}} \frac{1}{\langle \xi \rangle^{d+1}}
m(x) |\overline{g}^{(\alpha-\beta)}(t-x)|\,dx d\xi \\
\lesssim
C \cdot C_{f,g} ^{|\beta |^\s} |\beta |^{\t |\beta|^\s}
m(t)  \int_{\mR^{d}} v(t-x)|\overline{g}^{(\alpha-\beta)}(t-x)|\,dx  \\
= C \cdot C_{f,g} ^{|\beta|^\s} |\beta |^{\t |\beta|^\s}
m(t)  \| g^{(\alpha-\beta)}\|_{L^1 _v}, \quad t\in \Ren,
\end{multline*}
where $C$ depends on $\t,\s$ and $d$, and we used $\overline{(M.2)'}$ property of the sequence $p^{\t p^{\s}}$, $p\in\N$, $\t>0$, $\s>1$, cf. Lemma \ref{osobineM_p_s}.

Therefore, by  \eqref{uslovTeoreme} we obtain
\begin{multline*}
|\partial^{\alpha} f(t)|
\lesssim
\frac{m(t)}{\|g\|^2 _{L^2} }
\sum_{\beta\leq \alpha}{\alpha\choose \beta} (2\pi )^{|\beta|}\tilde C \cdot C_{f,g} ^{|\beta|^\s} |\beta |^{\t |\beta|^\s}
  \| g^{(\alpha-\beta)}\|_{L^1 _v}, \\
\lesssim
\tilde C  \frac{m(t)}{\|g\|^2 _{L^2} }
\sum_{\beta\leq \alpha}{\alpha\choose \beta} (2\pi )^{|\beta|} C_{f,g} ^{|\beta|^\s} |\beta |^{\t |\beta|^\s}
C_{g} ^{|\alpha - \beta|^\s} |\alpha -\beta |^{\t |\alpha -\beta|^\s}, \\
  \lesssim
 C  \frac{m(t)}{\|g\|^2 _{L^2} }
\tilde C_{f,g} ^{|\alpha|^\s} |\alpha |^{\t |\alpha|^\s},  \quad  t\in \Ren,
\end{multline*}
which gives \eqref{Eq:f-estimate}.
Here above we used \eqref{simpleInequality1} in several occasions.

\par

The equivalence between ii) and iii) follows immediately  from Definition \ref{De:asocirana}. Details are left for the reader.
\end{proof}

\par


Note that the condition \eqref{uslovTeoreme} is weaker than the corresponding condition in
\cite[Theorem 3.1]{CNR}, so by  \cite[Proposition 3.2]{CNR} one can choose elements from Gelfand-Shilov spaces as window functions
(both in Roumieu and Beurling case).

We note that if $ f \in {\mathcal D}^K_{\t,\s}$, $ \t >1,$ $ \s \geq 1$, then it obviously satisfies the condition i) in Theorem \ref{glavnaTeoremaSTFT},
i.e.
$$
|\partial^{\alpha} f(x)|\lesssim  m(x) C^{|\alpha|^{\s}}|\alpha|^{\t |\alpha|^{\s}},\quad \alpha\in \N^d,
$$
so that Theorem \ref{glavnaTeoremaSTFT} gives the decay properties of the STFT  of elements from $ {\mathcal D}^K_{\t,\s}$.

We use this remark to extend Theorem \ref{PosledicaPaley}. Recall the Paley-Wiener type result for
$ f\in {\mathcal D}^K_{\t,\s}$ describes the decay properties of the Fourier-Laplace transform in the context of the extended Gevrey regularity.
The role of compact support in  Paley-Wiener type theorems is essential.
In the following Corollary we weaken the assumptions from  \cite{PTT-04} (see also  \cite[Corollary 3.2]{PTT-05})
and allow the global growth condition given by \eqref{Eq:f-estimate}. Then, instead of cut-off functions, which are usually used in localization procedures,
we take a window in $ M^{1,1} _{v\otimes 1}(\Ren) \backslash\{0\}$, and give a  Paley-Wiener type result by using the  STFT.

\begin{cor}
\label{PosledicaPaley}
Let $\t>0$, $1<\s<2$, let $v$ be  a submultiplicative weight,  $m\in \cM _v$,
and let $g\in M^{1,1} _{v\otimes 1} (\Ren) \backslash\{0\}$ such that \eqref{uslovTeoreme} holds for some $ C_g>0$.
Then  a smooth function $f$ satisfies \eqref{Eq:f-estimate} if and only if
$$
|V_{g}f(x,\xi)|\lesssim m(x) \exp\Big\{-\Big(\frac{\s-1}{\t \s}\Big)^{\frac{1}{\s-1}}\frac{{{\ln}}^{\frac{\s}{\s-1}}(e+|\xi|)}{\ln^{\frac{1}{\s-1}}( \ln (e+|\xi|))}\Big\},\quad x,\xi \in \Ren.
$$
\end{cor}

The proof is an immediate consequence of Theorems \ref{propozicija} and \ref{glavnaTeoremaSTFT}, and Remark \ref{RemarkAsimtotska}.

\par

We finish this section with a version of Theorem \ref{glavnaTeoremaSTFT} and Corollary \ref{PosledicaPaley}
in the context of Beurling type ultradifferentiable functions.
The proofs are left as an exercise.

\begin{te}
\label{glavnaTeoremaSTFTBeurling}
Let $\t>0$, $\s\geq 1$, let $v$ be  a submultiplicative weight,  $m\in \cM _v$,
and let $g\in M^{1,1} _{v\otimes 1} (\Ren) \backslash\{0\}$ such that for some $C>0$,
$$
\|\partial^{\alpha} g\|_{L^1 _v (\Ren)}\leq C^{|\alpha|^{\s}+1}|\alpha|^{\t |\alpha|^{\s}},\quad \alpha \in {\bf N}^d.
$$
For a smooth function $f$ the following conditions are equivalent:
\begin{itemize}
\item[i)] For every $h>0$ there exists $A>0$ such that
\be
\label{Eq:f-estimate-Beurling}
 \|\partial^{\alpha} f\|_{L^{\infty} (\Ren)}\leq  m(x) A h^{|\alpha|^{\s}}|\alpha|^{\t |\alpha|^{\s}}, \,\, \alpha \in {\bf N}^d;
\ee
\item[ii)] For  every $h>0$ there exists $A>0$ such that
$$ |\xi|^{|\alpha|}|V_{g}f(x,\xi)| \leq   m(x) A h^{|\alpha|^{\s}}|\alpha|^{\t |\alpha|^{\s}},\,\,  x,\xi\in \Ren, \alpha \in {\bf N}^d; $$
\item[iii)] For  every $h>0$ there exists $A>0$ such that
$$ |V_{g}f(x,\xi)|\leq    m(x) A e^{-T_{\t,\s,h}(|\xi|)}, \,\, x,\xi\in \Ren. $$
\end{itemize}
\end{te}

\begin{cor}
\label{PosledicaPaleyBeurling}
Let $\t>0$, $1<\s<2$, and let $g\in M^{1,1}(\Ren) \backslash\{0\}$ satisfies condition \eqref{uslovTeoreme}.
Then a smooth function  $f$ satisfies \eqref{Eq:f-estimate-Beurling}
if and only if for every $H>0$ there exists $A>0$ such that
$$
|V_{g}f(x,\xi)|\leq  m(x) A \exp\Big\{-\Big(\frac{\s-1}{\t \s}\Big)^{\frac{1}{\s-1}}\frac{{{\ln}}^{\frac{\s}{\s-1}}(e+|\xi|)}{W^{\frac{1}{\s-1}}( H \ln (e+|\xi|))}\Big\}, \;\;\; x,\xi \in \Ren.
$$
\end{cor}

\begin{remark}
A more general versions of Corollaries \ref{PosledicaPaley} and \ref{PosledicaPaleyBeurling} when $\s \geq 2$ can be proved
by using the asymptotic formulas \eqref{nejednakostzaTeoremu1} from Theorem \ref{propozicija}. This will give different necessary and sufficient conditions for $f$
in terms of the decay properties of the STFT. We leave details for the reader.
\end{remark}

\section{Wave front sets $\WF_{\t,\s}$ and STFT }\label{secWFstft}

In this section we characterize wave front sets related to the classes introduced in Subsection \ref{secTest}, by the means of the STFT and extended associated function from Subsection \ref{secAsocirana}. We start with the following definition of the wave front set
$\WF _{\t,\s} (u)$ of a distribution $u$ with respect to the extended Gevrey regularity,
see also \cite{PTT-02, PTT-03, PTT-04, PTT-05, TT} for details.

\begin{definition}
\label{DefWF}
Let $U\subseteq \Ren$ be open, $\t>0$, $\s>1$ or $ \t >1 $ and $\s = 1$, $ u \in \mathcal{D}' (U)$, and let $(x_0,\xi_0)\in \Ren\times \Ren\backslash\{0\}$. Then $(x_0,\xi_0)\not \in \WF_{\{\t,\s\}}(u)$ (resp. $(x_0,\xi_0)\not \in \WF_{(\t,\s)}(u)$)  if and only if there exists a conic
neighborhood $\Gamma$ of $\xi_0$, a compact neighborhood
$ K $ of $x_0$, and
$\phi\in \D_{\{\t,\s\}}^K$ (resp. $\phi\in \D_{(\t,\s)}^K$ ) such that $\phi=1$ on some neighborhood of $x_0$, and there exists $A,h>0$ (for every $h>0$ there exists $A>0$ such that)
$$
|\widehat{\phi u}(\xi)|\leq A  \frac{h^{N^{\s}} N^{\t N^{\s}}}{|\xi|^N},\quad N\in {\N}\,,\xi\in \Gamma\,.
$$
\end{definition}

By using the Paley-Wiener theorem for  $ \D_{\t,\s}^K$ it can be proved that Definition \ref{DefWF} does not depend on the choice of the cut-off
function $\phi\in \D_{\t,\s}^K$, see  \cite{PTT-04}.

Note that when $ \t >1 $ and $\s = 1$ we have $\WF_{\{\t,1\}}(u) = \WF_\t (u),$ where $ \WF_\t (u)$ denotes Gevrey wave front set, cf. \cite{Rodino}.
We refer to \cite{PTT-02} for a relation between $ \WF_{\t,\s}(u)$ from Definition \ref{DefWF} and classical, analytic and Gevrey wave front sets.

By using the ideas presented in \cite{PrangoskiPilip} we resolve
$\WF _{\t,\s} (u)$ of a distribution $u$  via decay estimates of its STFT as follows.

\begin{te}
\label{NezavisnostWFThm}
Let $u\in \D'(\Ren)$, $\t>0$, $\s>1$. The following assertions are equivalent:
\begin{itemize}
\item[i)] $(x_0,\xi_0)\not\in \WF_{\{\t,\s\}}(u)$ (resp. $(x_0,\xi_0)\not \in \WF_{(\t,\s)}(u)$) .
\item[ii)] There exists a conic
neighborhood  $\Gamma$ of $\xi_0$, a compact neighborhood
$ K $ of $x_0$ such that for every $\phi\in \D_{\{\t,\s\}}^K$  (resp. $\phi\in \D_{(\t,\s)}^K$ )  there exists $A,h>0$ (resp. for every $h>0$ there exists $A>0$) such that
\be
\label{TeoremaWFUslov}
|\widehat{\phi u}(\xi)|\leq A  \frac{h^{N^{\s}} N^{\t N^{\s}}}{|\xi|^N},\quad N\in \N,\,\xi\in \Gamma\,;
\ee
\item[iii)]  There exists a conic
neighborhood  $\Gamma$ of $\xi_0$, a compact neighborhood
$ K $ of $x_0$ such that for every $ \phi\in \D_{\{\t,\s\}}^{K-\{x_0\}}$ (resp. $ \phi\in \D_{(\t,\s)}^{K-\{x_0\}}$) there exists $A,h>0$ (resp. for every $h>0$ there exists $A>0$) such that
\be
\label{TeoremaWFUslov2}
\dss |V_{\phi}u(x,\xi)|\leq A e^{-T_{\t,\s,h}(|\xi|)},\quad  x\in K,\,\xi\in \Gamma,
\ee  where $K-\{x_0\}=\{y\in \Ren\,|\, y+x_0\in K\}$.
\end{itemize}
\end{te}

\begin{proof}
We give the proof for the Roumieu case and leave the Beurling case to the reader.

The equivalence i) $ \Leftrightarrow $ ii) is proved in Theorem 4.2. in \cite{PTT-04}.

ii) $\Rightarrow$ iii) Without loss of generality we may assume that there exists a conic
neighborhood  $\Gamma$ of $\xi_0$, compact neighborhood $K_1=\overline{B_r (x_0)}$, $r>0$, such that for every $\phi\in \D_{\t,\s}^{K_1}$ \eqref{TeoremaWFUslov} holds.

Set $K=\overline{B_{r/2} (x_0)}$ and note that if $\phi (t)$ is an arbitrary function $ \D_{\t,\s}^{K-\{x_0\}}$ and $x\in K$, then $T_x \phi (t)$ is a function in $ \D_{\t,\s}^{K_1}$.

Using the definition of STFT and  \eqref{TeoremaWFUslov} we have that
$$
|V_{\phi}u(x,\xi)|=|\Fur ({u T_x\phi })|\leq A \inf_{N\in \N}\frac{h^{N^{\s}}N^{\t N^{\s}}}{|\xi|^N}= A e^{-T_{\t,\s,1/h}(|\xi|)},
$$
$x\in K$, $\xi\in \Gamma,$ for some constant $A>0$, and iii) follows.

iii)  $\Rightarrow$ i) Since $u\in \D'(\Ren)$, we may choose the window function $\phi\in \D_{\t,\s}(\Ren)$ in iii) to be centered near any point in $\Ren$. Let $\phi$ be centered near $0$, then clearly $\psi=T_{x_0}\phi$ is centered near $x_0$ and \eqref{TeoremaWFUslov2} implies
$$
|\widehat{\psi u}(\xi)|=|V_{\phi} u (x_0,\xi)|\lesssim   e^{-T_{\t,\s,h}(|\xi|)}\lesssim \inf_{N\in \N}\frac{(1/h)^{N^{\s}}N^{\t N^{\s}}}{|\xi|^N},\quad \xi \in \Gamma,
$$
for some $h>0$ and the proof is finished.
\end{proof}

Next we discuss local extended Gevrey regularity via the STFT. To that end we introduce the singular support as follows (cf. \cite{TT0}).

\begin{definition}
Let  there be given $x_0\in \Ren$, $u\in \D'(U)$,  $\t>0$ and $\s>1$. Then $x_0\not \in \sing_{\tau,\s}(u)$ if and only if there exists open neighborhood $\Omega \subset U$ of $x_0$ such that $u\in \E_{\t,\s}(\Omega)$.
\end{definition}

The local regularity is related to the wave front set as follows.

\begin{proposition} \label{projection}
Let $\t>0$ and $\s>1$, $u\in \D'(U)$. Let $\pi_1:U\times\Ren \backslash \{0\}\to U$ be the standard projection given by $\pi_1(x,\xi)=x$. Then
\be
\sing_{\t,\s}(u)=\pi_1(\WF_{\t,\s}(u))\,.\nonumber
\ee
\end{proposition}

We refer to \cite{TT0} for the proof, see also \cite[Theorem 3.1]{PTT-02} and \cite[Proposition 11.1.1]{FJ}.

\begin{te}
\label{NezavisnostWFThm2}
Let  there be given $x_0\in \Ren$, $u\in \D'(U)$,  $\t>0$ and $\s>1$.  Then $x_0\not \in \sing_{\tau,\s}(u)$  if and only if there exists a compact neighborhood
$ K $ of $x_0$ such that for every
$ \phi\in \D_{\{\t,\s\}}^{K-\{x_0\}}$ (resp. $ \phi\in \D_{(\t,\s)}^{K-\{x_0\}}$) there exists $A,h>0$ (resp. for every $h>0$ there exists $A>0$) such that
\be
\label{TeoremaWFUslov3}
\dss |V_{\phi}u(x,\xi)|\leq A  e^{-T_{\t,\s,h}(|\xi|)},\quad x\in K,\,\xi\in \Ren\backslash\{0\},
\ee
where $K-\{x_0\}=\{y\in \Ren\,|\, y+x_0\in K\}$.
\end{te}

\begin{proof} We give the proof for the Roumieu case only.
If $ x_0 \not\in \sing_{\{\tau,\s\}}(u)$ then  by Proposition \ref{projection} it follows that
$ (x_0, \xi) \not\in  \WF_{\{\t,\s\}}(u) $ for any $ \xi\in \Ren\backslash\{0\}$, so that \eqref{TeoremaWFUslov3} follows from
Theorem \ref{NezavisnostWFThm} iii).

Now assume that  \eqref{TeoremaWFUslov3} holds, then it holds for any cone $\Gamma $. By Theorem \ref{NezavisnostWFThm} we conclude that
$ (x_0, \xi) \not\in  \WF_{\{\t,\s\}}(u) $ for any $ \xi\in \Ren\backslash\{0\}.$ Now Proposition \ref{projection} implies that
$x_0 \not\in  \sing_{\{\tau,\s\}}(u)$, and the proof is completed.
\end{proof}

\section*{Acknowledgement}
This work is supported by MPNTR through Project 174024.

\end{document}